\documentclass[11pt,oneside,reqno]{amsart}

\usepackage[a4paper, total={450pt,675pt}]{geometry}
\usepackage[OT2,T1]{fontenc}
\usepackage[utf8]{inputenc}
\usepackage{amssymb,bm}

\usepackage[dvipsnames]{xcolor}
\usepackage[
colorlinks=true,
linkcolor=Maroon,
citecolor=JungleGreen,
urlcolor=NavyBlue]{hyperref}

\usepackage[capitalize,nameinlink]{cleveref}

\usepackage{enumitem}
\usepackage{dsfont}
\usepackage{subcaption}
\usepackage{centernot}

\usepackage{tikz}
\usetikzlibrary{arrows}
\usetikzlibrary{decorations.pathreplacing,angles,quotes}
\usepackage{pgfplots}
\pgfplotsset{compat=1.15}
\usepackage{mathrsfs}
\usetikzlibrary{arrows}

\usepackage{soul}
\setstcolor{red}
\makeatletter
\def\@tocline#1#2#3#4#5#6#7{\relax
	\ifnum #1>\c@tocdepth % then omit
	\else
	\par \addpenalty\@secpenalty\addvspace{#2}%
	\begingroup \hyphenpenalty\@M
	\@ifempty{#4}{%
		\@tempdima\csname r@tocindent\number#1\endcsname\relax
	}{%
		\@tempdima#4\relax
	}%
	\parindent\z@ \leftskip#3\relax
	\advance\leftskip\@tempdima\relax
	\rightskip\@pnumwidth plus4em \parfillskip-\@pnumwidth
	#5\leavevmode\hskip-\@tempdima
	\ifcase #1
	\or\or \hskip 2em \or \hskip 2em \else \hskip 3em \fi%
	#6\nobreak\relax
	\dotfill\hbox to\@pnumwidth{\@tocpagenum{#7}}\par
	\nobreak
	\endgroup
	\fi} 
\makeatother

\usepackage[
backend=bibtex,
style=alphabetic,
maxbibnames=10,
sorting=nyt]{biblatex}

\DeclareLabelalphaTemplate{
	\labelelement{
		\field[final]{shorthand}
		\field{label}
		\field[strwidth=3]{labelname}}
	\labelelement{
		\field[strwidth=2,strside=right]{year}}
}

\DeclareFieldFormat[article,book,incollection,misc]{title}{#1}
\DeclareFieldFormat[inbook,incollection]{booktitle}{#1}
\DeclareFieldFormat[misc]{date}{\textit{preprint} (#1)}
\DeclareFieldFormat{pages}{#1}
\DeclareFieldFormat{doi}{
	\href{http://www.ams.org/mathscinet-getitem?mr=#1}{#1}}
\renewbibmacro{in:}{%
	\ifentrytype{article}{}{\printtext{\bibstring{in}\intitlepunct}}}

\newtheorem{theorem}{Theorem}[section]
\newtheorem{definition}[theorem]{Definition}
\newtheorem{lemma}[theorem]{Lemma}
\newtheorem{proposition}[theorem]{Proposition}
\newtheorem{corollary}[theorem]{Corollary}

\newtheorem{remark}[theorem]{Remark}

\crefname{section}{Sect.}{section}
\numberwithin{equation}{section}
\newcommand*\diff{\mathop{}\!\mathrm{d}}

\DeclareMathOperator{\supp}{supp}

\DeclareMathOperator{\const}{const.}

\begin{document}
	%%%%%%%%%%%%%%%%%%%%%%%%%%%%%%%%%%%%%%%%%%%%%%%%%%%%%%%%%%%%%%%%%%%%%%%%%%%%%%%%
	%%%%%%%%%%%%%%%%%%%%%               TITLE PAGE               %%%%%%%%%%%%%%%%%%%
	%%%%%%%%%%%%%%%%%%%%%%%%%%%%%%%%%%%%%%%%%%%%%%%%%%%%%%%%%%%%%%%%%%%%%%%%%%%%%%%%
	\title[Asymptotics for the infinite Brownian loop on noncompact 
	sym. spaces]{Asymptotics for the infinite Brownian loop on noncompact 
		symmetric spaces}
	
	\author{Effie Papageorgiou}
	
	\begin{abstract}
	The infinite Brownian loop  on a Riemannian manifold is the limit in distribution of the Brownian bridge of length $T$
	around a fixed origin when $T \rightarrow +\infty$.	The aim of this note is to study its long-time asymptotics on Riemannian symmetric spaces  $G/K$ of noncompact type and of general rank.
		This amounts to the behavior of solutions
		to the heat equation subject to the Doob transform induced by the ground spherical function.
		Unlike the standard Brownian motion, we observe in this
		case phenomena which are similar to the
		Euclidean setting, namely $L^1$ asymptotic convergence without requiring bi-$K$-invariance
		for initial data, and strong $L^{\infty}$ convergence.
	\end{abstract}
	
	
	%%%%%%%%%%% MSC2020
	\makeatletter
	\@namedef{subjclassname@2020}{\textnormal{2020}
		\it{Mathematics Subject Classification}}
	\makeatother
	%%%%%%%%%%% MSC2020
	
	\maketitle
	\tableofcontents
	%%%%%%%%%%%%%%%%%%%%%%%%%%%%%%%%%%%%%%%%%%%%%%%%%%%%%%%%%%%%%%%%%%%%%%%%%%%%%%%%
	%%%%%%%%%%%%%%%%%%%%%                SECTION I               %%%%%%%%%%%%%%%%%%%
	%%%%%%%%%%%%%%%%%%%%%%%%%%%%%%%%%%%%%%%%%%%%%%%%%%%%%%%%%%%%%%%%%%%%%%%%%%%%%%%%
	\section{Introduction}\label{Section.1 Intro}
	
	The heat equation is one of the most fundamental partial differential equations 
	in mathematics. 
	It has been extensively studied in various settings and is known to play a
	central role in several areas of mathematics (see for instance \cite{Gri2009}).
	
	Let $\mathcal{M}$ be a complete non-compact Riemannian manifold. Let $(X_t, \mathbb{P}_x)$ be the Brownian motion
	on $\mathcal{M}$, that is, the stochastic process generated by the Laplace-Beltrami operator $\Delta$. Let also
	$h_t(x, y)$ be the heat kernel on $\mathcal{M}$, that is, the minimal positive fundamental solution of the heat
	equation $\partial_t u = \Delta u$ on $(0, \infty) \times \mathcal{M}$. Then $h_t(x,y)$ is also the transition density of $X_t$, which
	means that for any Borel set $A \subset\mathcal{M}$, 
	$$\mathbb{P}_x(X_t \in A) =\int_A \diff{\mu}(y)\,
	h_t( x, y)$$
	where $\diff{\mu}(y)$ denotes the Riemannian measure.

	The following classical long-time asymptotic convergence result
	corresponds to the {\it Central Limit Theorem} of probability in the PDE setting. 
	We refer to the expository survey \cite{Vaz2018} for more details on this property. 
	
	\begin{theorem}
		Consider the heat equation
		\begin{align}\label{S1 HE intro}
			\begin{cases}
				\partial_{t}u(t,x)\,
				=\,\Delta_{\mathbb{R}^{n}}u(t,x),
				\qquad\,t>0,\,\,x\in\mathbb{R}^{n}\\[5pt]
				u(0,x)\,=\,f(x),
			\end{cases}
		\end{align}
		where the initial data $f$ belongs to $L^{1}(\mathbb{R}^n)$.
		Denote by $M=\int_{\mathbb{R}^n}\diff{x}\,f(x)$ the mass of $f$
		and by $G_{t}(x,y)\,=\,(4\pi{t})^{-n/2}e^{-|x-y|^{2}/4t}$ the heat kernel.
		Then the solution to \eqref{S1 HE intro} satisfies:
		\begin{align} \label{S1 L1 R}
			\|u(t,\,\cdot\,)-MG_{t}(\cdot, x_0)\|_{L^{1}(\mathbb{R}^n)}\,
			\longrightarrow\,0
		\end{align}
		and
		\begin{align}\label{S1 Linf R}
			t^{\frac{n}{2}}\,\|u(t,\,\cdot\,)-MG_{t}(\cdot, x_0)\|_{L^{\infty}(\mathbb{R}^n)}\,
			\longrightarrow\,0
		\end{align}
		as $t\rightarrow\infty$. The $L^p$ ($1<p<\infty$) norm estimates 
		follow by interpolation.
	\end{theorem}
	
	The convergence properties (\ref{S1 L1 R}) and (\ref{S1 Linf R}) have an
	interesting probabilistic meaning,  that the Brownian motion eventually
	\textquotedblleft forgets\textquotedblright\ about its starting point $%
	x_{0}, $ which corresponds to the fact that $X_{t}$ escapes to $\infty $
	rotating chaotically in angular direction. Similar properties are true in the case of volume doubling Riemannian manifolds as long as some properties of the heat kernel are satisfied (e.g. two-sided heat kernel estimates of the Li-Yau type, as is true on manifolds with nonnegative Ricci curvature) but surprisingly fail in certain manifolds with ends, see \cite{GPZ22}.
	
	The situation is drastically different in hyperbolic spaces. It was shown by
	V\'{a}zquez \cite{Vaz2019} that (\ref{S1 L1 R}) fails for a general initial
	function $u_{0}\in L^{1}\left( \mathbb{H}^{n}\right) $, even if compactly supported, but is still true if $%
	u_{0}$ is spherically symmetric around $x_{0}.$   Recall that in hyperbolic spaces  
	Brownian motion $X_{t}$ tends to escape to $\infty $ along geodesics, which
	means that it \textquotedblleft remembers\textquotedblright\ at least the
	direction of the starting point $x_{0}.$ For full generalizations of the previous result in Riemannian symmetric spaces of noncompact type, we refer to 
	the recent paper \cite{APZ2023},
	where  some arguments in \cite{Vaz2019} are also clarified. Note that these spaces
	have nonpositive sectional curvature. Besides the Laplace-Beltrami operator, a second Laplacian is considered in \cite{APZ2023}, the so-called distinguished Laplacian. In this paper we consider the Doob transform of the Brownian motion arising by the unique bi-$K$-invariant  ground spherical function. This new process is the limit in distribution of the Brownian bridge of length $L$
	around a fixed origin $x_o$, when $L\rightarrow +\infty$.  Let us elaborate.
	
	Anker, Bougerol and Jeulin introduced in \cite{ABJ02} the concept of the ``infinite Brownian loop" on Riemannian manifolds, which is roughly speaking the limit (if it exists) of the
	Brownian motion constrained to come back to its starting point at a very
	large time (see Section \ref{Section.2 Prelim} for a precise definition).  Often, the infinite Brownian loop is the Brownian motion itself (e.g. when the manifold is of nonnegative Ricci curvature), but this is not the case on non-compact symmetric spaces	$\mathbb{X}=G/K$ of dimension $n\geq 2$. To describe this new Markov process, let us introduce some notation. Let
	$\Delta$ be the Laplace-Beltrami operator on the Riemannian manifold $(\mathbb{X}, \mu)$ and $h_t$ be the associated heat kernel. Then, the bottom of the spectrum is equal to $|\rho|^2$ and let $\varphi_{0}$ be the ground spherical function, which is positive.
	Consider the new manifold $(\mathbb{X}, \widetilde{\mu})$, where $\diff{\widetilde{\mu}}=\varphi_{0}^2 \,\diff{\mu}$. 
	Denote by $\Delta_{\widetilde{\mu}}$ the new, relativized Laplacian and by $\widetilde{{h}_{t}}=e^{|\rho|^{2}t}h_{t}/\phi_0$
	the associated heat kernel. Then, the infinite Brownian loop  is exactly the relativized
	$\varphi_0$-process.  
	
	Consider the Cauchy problem
	\begin{align}\label{S1 HE S}
		\partial_{t}u(t,gK)\,
		=\,\Delta_{\widetilde{\mu}}u(t,gK),
		\qquad
		{u}(0,gK)\,=\,f(gK).
	\end{align}
	Denote by 
	$\widetilde{M}=\tfrac{\frac{f}{\varphi_0}*{\varphi}_{0}}{\varphi_{0}}$ the mass function
	on $\mathbb{X}$ which generalizes the mass in the Euclidean case (see \cref{rmk mass}). Here, the convolution is realized on $(\mathbb{X}, \widetilde{\mu})$ (with a slight but obvious abuse of notation).
	Then, we show the following long-time asymptotic convergence results.
	\begin{theorem}\label{S1 Main thm 2}
		Let $f$ be continuous and compactly supported initial data on
		$\mathbb{X}$. Then, the solution to the heat equation \eqref{S1 HE S} satisfies
		\begin{align}\label{S1 L1 disting}
			\|u(t,\,\cdot\,)-
			\widetilde{M}\,\widetilde{h_t}\|_{L^{1}(\widetilde{\mu})}\,
			\longrightarrow\,0
		\end{align}
		and
		\begin{align}\label{S1 Linf disting}
			t^{\frac{\nu+n}{4}}
			\|u(t,\,\cdot\,)-
			\widetilde{M}\,\widetilde{h_t}\|_{L^{\infty}(\widetilde{\mu})}\,
			\longrightarrow\,0
		\end{align}
		as $t\rightarrow\infty$. Here $\nu$ denotes the dimension at infinity of $G/K$.
		Analogous $L^p$ ($1<p<\infty$) norm estimates 
		follow by interpolation.
	\end{theorem}
	
	\begin{remark}
		Let us comment about \eqref{S1 L1 disting} and \eqref{S1 Linf disting}.
		Firstly, the $L^1$ convergence \eqref{S1 L1 disting} holds 
		without the bi-$K$-invariance restriction on initial data, as in the case of the distinguished Laplacian and in contrast to the case of the Laplace-Beltrami operator, \cite{APZ2023}.  Also,  
		the sup norm estimate $\eqref{S1 Linf disting}$ resembles the Euclidean setting and the case of the distinguished Laplacian $\Delta_d$. This strong convergence is not true in the case of the Laplace-Beltrami operator, \cite[Remark 3.6]{APZ2023}. Secondly, notice that the heat kernel $\widetilde{h_t}$ is bi-$K$-invariant as in the case of the semigroup $e^{t\Delta}$, but not in the case of $e^{t\Delta_d}$.
		Yet, the mass $\widetilde{M}$ is a bounded function and no more a constant in general.
		Thirdly, the power $(\nu+n)/4$ which occurs in the time factor, is always larger than $(\ell+n)/4$ appearing in the distinguished Laplacian case, \cite[Theorem 1.6]{APZ2023}.  
		Indeed, here $\ell$ denotes the rank of $G/K$ and $\Sigma_{r}^{+}$ the set of positive reduced roots, and
		$\nu=\ell+2|\Sigma_{r}^{+}|$ the dimension at infinity. For example in real hyperbolic space, we have $\ell=1$, $|\Sigma_{r}^{+}|=1$, $\nu=3$.
	\end{remark}

	\begin{remark}
		The asymptotic convergences \eqref{S1 L1 disting} and \eqref{S1 Linf disting}
		hold for some larger classes of initial data, see \cref{Subsect other data}.
		These classes are not optimal and finding the right function space is
		an interesting question for further study.
	\end{remark}
	
	Apart from the results themselves, we want to emphasize the strategy to solving such problems, as it may be useful for settings where one can get delicate results for the heat kernel such as asymptotics (e.g. Damek-Ricci spaces, homogeneous trees, or even for Doob transforms of the Euclidean heat kernel). In this sense, one has to compare a quotient of two heat kernels inside the critical region (see \cref{S4 Lemma ratios difference})  to get information for the asymptotic behavior of the solutions to the heat equation (see \cref{section3}). For instance, it is essentially the large time asymptotics of this quotient that show what the mass should be and how it behaves under different initial data.
	
	This paper is organized as follows. After the present introduction
	in \cref{Section.1 Intro} and preliminaries in \cref{Section.2 Prelim}, 
	we deal with the long-time asymptotic behavior of solutions to the heat equation
	associated with the relativized Laplacian  on symmetric spaces in 
	\cref{section3}, arising as a specific Doob transform. 
	After specifying the critical region in this context, we study the long-time
	convergence in $L^1$ and in $L^{\infty}$ with compactly supported initial
	data. Questions associated with other initial data are discussed at the end 
	of the paper.
	
	Throughout this paper, the notation
	$A\lesssim{B}$ between two positive expressions means that 
	there is a constant $C>0$ such that $A\le{C}B$. 
	The notation $A\asymp{B}$ means that $A\lesssim{B}$ and $B\lesssim{A}$.
	%%%%%%%%%%%%%%%%%%%%%%%%%%%%%%%%%%%%%%%%%%%%%%%%%%%%%%%%%%%%
	
	%%%%%%%%%%%%%%%%%%%%%%%%%%%%%%%%%%%%%%%%%%%%%%%%%%%%%%%%%%%%%%%%%%%%%%%%%%%%%%%%
	%%%%%%%%%%%%%%%%%%%%%               SECTION II               %%%%%%%%%%%%%%%%%%%
	%%%%%%%%%%%%%%%%%%%%%%%%%%%%%%%%%%%%%%%%%%%%%%%%%%%%%%%%%%%%%%%%%%%%%%%%%%%%%%%%
	\section{Preliminaries}\label{Section.2 Prelim}
	In this section, we review spherical Fourier analysis 
	on Riemannian symmetric spaces of noncompact type. The notation is standard 
	and follows \cite{Hel1978,Hel2000,GaVa1988}. 
	Next we recall bounds and asymptotics of the heat kernel, for which 
	we refer to \cite{AnJi1999,AnOs2003} for more details in this setting.  We next recall the definition of the Doob transform on general Riemannian manifolds, following \cite{Gri2009} and finally we describe the infinite Brownian loop, for which we refer to \cite{ABJ02}.
	
	\subsection{Noncompact Riemannian symmetric spaces}
	Let $G$ be a semi-simple Lie group, connected, noncompact, with finite center, 
	and $K$ be a maximal compact subgroup of $G$. The homogeneous space 
	$\mathbb{X}=G/K$ is a Riemannian symmetric space of noncompact type.
	Let $\mathfrak{g}=\mathfrak{k}\oplus\mathfrak{p}$ be the Cartan decomposition 
	of the Lie algebra of $G$. The Killing form of $\mathfrak{g}$ induces 
	a $K$-invariant inner product $\langle\,.\,,\,.\,\rangle$ on $\mathfrak{p}$, 
	hence a $G$-invariant Riemannian metric on $G/K$.
	We denote by $d(\,.\,,\,.\,)$ the Riemannian distance on $\mathbb{X}$.
	
	Fix a maximal abelian subspace $\mathfrak{a}$ in $\mathfrak{p}$. 
	The rank of $\mathbb{X}$ is the dimension $\ell$ of $\mathfrak{a}$.
	We identify $\mathfrak{a}$ with its dual $\mathfrak{a}^{*}$ 
	by means of the inner product inherited from $\mathfrak{p}$.
	Let $\Sigma\subset\mathfrak{a}$ be the root system of 
	$(\mathfrak{g},\mathfrak{a})$ and denote by $W$ the Weyl group 
	associated with $\Sigma$. 
	Once a positive Weyl chamber $\mathfrak{a}^{+}\subset\mathfrak{a}$ 
	has been selected, $\Sigma^{+}$ (resp. $\Sigma_{r}^{+}$ 
	or $\Sigma_{s}^{+}$)  denotes the corresponding set of positive roots 
	(resp. positive reduced, i.e., indivisible roots or simple roots).
	Let $n$ be the dimension and $\nu$ be the pseudo-dimension 
	(or dimension at infinity) of $\mathbb{X}$: 
	\begin{align}\label{S2 Dimensions}
		\textstyle
		n\,=\,
		\ell+\sum_{\alpha \in \Sigma^{+}}\,m_{\alpha}
		\qquad\textnormal{and}\qquad
		\nu\,=\,\ell+2|\Sigma_{r}^{+}|
	\end{align}
	where $m_{\alpha}$ denotes the dimension of the positive root subspace
	\begin{align*}
		\mathfrak{g}_{\alpha}\,
		=\,\lbrace{
			X\in\mathfrak{g}\,|\,[H,X]=\langle{\alpha,H}\rangle{X},\,
			\forall\,H\in\mathfrak{a}
		}\rbrace.
	\end{align*}
	
	Let $\mathfrak{n}$ be the nilpotent Lie subalgebra 
	of $\mathfrak{g}$ associated with $\Sigma^{+}$ 
	and let $N = \exp \mathfrak{n}$ be the corresponding 
	Lie subgroup of $G$. We have the decompositions 
	\begin{align*}
		\begin{cases}
			\,G\,=\,N\,(\exp\mathfrak{a})\,K 
			\qquad&\textnormal{(Iwasawa)}, \\[5pt]
			\,G\,=\,K\,(\exp\overline{\mathfrak{a}^{+}})\,K
			\qquad&\textnormal{(Cartan)}.
		\end{cases}
	\end{align*}
	Denote by $A(x)\in\mathfrak{a}$ and $x^{+}\in\overline{\mathfrak{a}^{+}}$
	the middle components of $x\in{G}$ in these two decompositions, and by
	$|x|=|x^{+}|$ the distance to the origin.
	In the Cartan decomposition, the Haar measure 
	on $G$ writes
	\begin{align*}
		\int_{G}\diff{x}\,f(x)\,
		=\,
		|K/\mathbb{M}|\,\int_{K}\diff{k_1}\,
		\int_{\mathfrak{a}^{+}}\diff{x^{+}}\,\delta(x^{+})\, 
		\int_{K}\diff{k_2}\,f(k_{1}(\exp x^{+})k_{2})\,,
	\end{align*}
	with density
	\begin{align}\label{S2 estimate of delta}
		\delta(x^{+})\,
		=\,\prod_{\alpha\in\Sigma^{+}}\,
		(\sinh\langle{\alpha,x^{+}}\rangle)^{m_{\alpha}}\,
		\asymp\,
		\prod_{\alpha\in\Sigma^{+}}
		\Big( 
		\frac{\langle\alpha,x^{+}\rangle}
		{1+\langle\alpha,x^{+}\rangle}
		\Big)^{m_{\alpha}}\,
		e^{2\langle\rho,x^{+}\rangle}
		\qquad\forall\,x^{+}\in\overline{\mathfrak{a}^{+}}. 
	\end{align}
	Here $K$ is equipped with its normalized Haar measure, $\mathbb{M}$ denotes the centralizer of $\exp\mathfrak{a}$ in $K$ and the volume 
	of $K/\mathbb{M}$ can be computed explicitly, see \cite[Eq (2.2.4)]{AnJi1999}.
	Recall that $\rho\in\mathfrak{a}^{+}$ denotes the half sum of all positive roots 
	$\alpha \in \Sigma^{+}$ counted with their multiplicities $m_{\alpha}$:
	\begin{align*}
		\rho\,=\,
		\frac{1}{2}\,\sum_{\alpha\in\Sigma^{+}} \,m_{\alpha}\,\alpha.
	\end{align*}
	
	%%%%%%%%%%%%%%%%%%%%%%%%%%%%%%%%%%%%%%%%%%%%%%%%%%%%%%%%%%%%%%%%%%%%%%%%%%%%%%%%
	\subsection{Spherical Fourier analysis} For this subsection, our main references are  \cite[Chap.4]{GaVa1988} and 
	\cite[Chap.IV]{Hel2000}.
	
	For every $\lambda\in\mathfrak{a}$, the spherical function
	$\varphi_{\lambda}$ is a smooth bi-$K$-invariant eigenfunction of all 
	$G$-invariant differential operators on $\mathbb{X}$, in particular of the
	Laplace-Beltrami operator:
	\begin{equation*}
		-\Delta\varphi_{\lambda}(x)\,
		=\,(|\lambda|^{2}+|\rho|^2)\,\varphi_{\lambda}(x).
	\end{equation*}
	It is symmetric in the sense that
	$\varphi_{\lambda}(x^{-1})=\varphi_{-\lambda}(x)$,
	and is given by the integral representation
	\begin{align}\label{S2 Spherical Function}
		\varphi_{\lambda}(x)\, 
		=\,\int_{K}\diff{k}\,e^{\langle{i\lambda+\rho,\,A(kx)}\rangle}.
	\end{align}
	
	All the elementary spherical functions $\varphi_{\lambda}$ 
	with parameter $\lambda\in\mathfrak{a}$ are controlled by the ground spherical
	function $\varphi_{0}$, which satisfies the global estimate
	\begin{align}\label{S2 global estimate phi0}
		\varphi_{0}(\exp{H})\,
		\asymp\,
		\Big\lbrace \prod_{\alpha\in\Sigma_{r}^{+}} 
		1+\langle\alpha,H\rangle\Big\rbrace\,
		e^{-\langle\rho, H\rangle}
		\qquad\forall\,H\in\overline{\mathfrak{a}^{+}}.
	\end{align}
	We will also use the following three properties, \cite[Proposition 4.6.3]{GaVa1988}: first,
	\begin{equation}\label{phi0split}
		\int_{K}\diff{k}\,\varphi_{0}(xky)=\varphi_{0}(x)\varphi_{0}(y).
	\end{equation} 
	Second, if $\Omega \subseteq G$ is any compact set, there is a constant $c = c(\Omega)\geq 1$ such that
	\begin{equation}\label{phi0bdd}
		c^{-1} \varphi_{0}(x)\leq\varphi_{0}(y_1xy_2)\leq c\, \varphi_{0}(x)\quad \forall x\in G, \, \forall y_1,y_2\in \Omega.
	\end{equation} 
	Finally,  it holds
	\begin{align}\label{S4 spherical split}
		\varphi_{0}(g^{-1}y)\,
		=\,\int_{K}\diff{k}\,e^{\langle{\rho,A(kg})\rangle}\,
		e^{\langle{\rho,A(ky})\rangle}
	\end{align}
	(see for instance \cite[Chap.III, Theorem 1.1]{Hel1994})
	
	Let $\mathcal{S}(K \backslash{G}/K)$ be the Schwartz space of bi-$K$-invariant
	functions on $G$. The spherical Fourier transform (Harish-Chandra transform)
	$\mathcal{H}$ is defined by
	\begin{align}\label{S2 HC transform}
		\mathcal{H}f(\lambda)\,
		=\,\int_{G}\diff{x}\,\varphi_{-\lambda}(x)\,f(x) 
		\qquad\forall\,\lambda\in\mathfrak{a},\
		\forall\,f\in\mathcal{S}(K\backslash{G/K}),
	\end{align}
	where $\varphi_{\lambda}\in\mathcal{C}^{\infty}(K\backslash{G/K})$ is the
	spherical function of index $\lambda \in \mathfrak{a}$.
	Denote by $\mathcal{S}(\mathfrak{a})^{W}$ the subspace 
	of $W$-invariant functions in the Schwartz space $\mathcal{S}(\mathfrak{a})$. 
	Then $\mathcal{H}$ is an isomorphism between $\mathcal{S}(K\backslash{G/K})$ 
	and $\mathcal{S}(\mathfrak{a})^{W}$. The inverse spherical Fourier transform 
	is given by
	\begin{align}\label{S2 Inverse formula}
		f(x)\,
		=\,C_0\,\int_{\mathfrak{a}}\,\diff{\lambda}\,
		|\mathbf{c(\lambda)}|^{-2}\,
		\varphi_{\lambda}(x)\,
		\mathcal{H}f(\lambda) 
		\qquad\forall\,x\in{G},\
		\forall\,f\in\mathcal{S}(\mathfrak{a})^{W},
	\end{align}
	where the constant $C_0=2^{n-\ell}/(2\pi)^{\ell}|K/\mathbb{M}||W|$ depends only 
	on the geometry of $\mathbb{X}$, and $|\mathbf{c(\lambda)}|^{-2}$ is the 
	so-called Plancherel density, given by an explicit formula by Gindikin-Karpelevič.

	%%%%%%%%%%%%%%%%%%%%%%%%%%%%%%%%%%%%%%%%%%%%%%%%%%%%%%%%%%%%%%%%%%%%%%%%%%%%%%%%
	\subsection{Heat kernel on symmetric spaces}
	The heat kernel on $\mathbb{X}$ is a positive bi-$K$-invariant right 
	convolution kernel, i.e., $h_{t}(xK,yK)=h_{t}(y^{-1}x)>0$, 
	which is thus determined by its restriction 
	to the positive Weyl chamber. 
	According to the inversion formula of the spherical Fourier transform, 
	the heat kernel is given by
	\begin{align}\label{S2 heat kernel inv}
		h_{t}(xK)\,
		=\,C_{0}\,\int_{\mathfrak{a}}\,\diff{\lambda}\,
		|\mathbf{c(\lambda)}|^{-2}\,
		\varphi_{\lambda}(x)\,
		e^{-t(|\lambda|^{2}+|\rho|^{2})}
	\end{align}
	and satisfies the global estimate
	\begin{align}\label{S2 heat kernel}
		h_{t}(\exp{H})\,
		\asymp\,t^{-\frac{n}{2}}\,
		\Big\lbrace{
			\prod_{\alpha\in\Sigma_{r}^{+}}
			(1+t+\langle{\alpha,H}\rangle)^{\frac{m_{\alpha}+m_{2\alpha}}{2}-1}
		}\Big\rbrace\,\varphi_{0}(\exp{H})
		e^{-|\rho|^{2}t-\frac{|H|^{2}}{4t}}
	\end{align}
	for all $t>0$ and $H\in\overline{\mathfrak{a}^{+}}$, 
	see \cite{AnJi1999,AnOs2003}. Recall that $\int_{\mathbb{X}}\diff{\mu}\,h_{t}=1$.
	
	Moreover, as long as $|H|=O(t)$, the large time
	behavior of the heat kernel can be described more accurately by the following
	asymptotics \cite[Theorem 5.1.1]{AnJi1999}:
	\begin{align}\label{S2 heat kernel critical region}
		h_t(\exp{H})\,
		\sim\,C_{1}\,t^{-\frac{\nu}{2}}\,
		\mathbf{b}\big(-i\tfrac{H}{2t}\big)^{-1}\,
		\varphi_{0}(\exp{H})\,e^{-|\rho|^{2}t-\frac{|H|^{2}}{4t}}
	\end{align}
	as $t\rightarrow\infty$. Here, the constant $C_1$ is given explicitly in \cite[Theorem 5.1.1]{AnJi1999}, as well the expression of the function $\mathbf{b}(-\lambda)^{\pm1}$, which is holomorphic for 
	$\lambda\in\mathfrak{a}+i\overline{\mathfrak{a}^{+}}$ and for 
	$\lambda\in i\overline{\mathfrak{a}^{+}}$ it is positive. We recall that it has the following
	behavior
	\begin{align}\label{bfunction}
		|\mathbf{b}(-\lambda)|^{-1}\,
		\asymp\,
		\prod_{\alpha\in\Sigma_{r}^{+}}\,
		(1+|\langle{\alpha,\lambda}\rangle|)^{\frac{m_{\alpha}+m_{2\alpha}}{2}-1}
	\end{align}
	and that its derivatives can be estimated by
	\begin{align}\label{bfunction derivative}
		p(\tfrac{\partial}{\partial\lambda})
		\mathbf{b}(-\lambda)^{-1}\,
		=\,\mathrm{O}\big(|\mathbf{b}(-\lambda)|^{-1}\big),
	\end{align}
	where $p(\tfrac{\partial}{\partial\lambda})$ is any differential 
	polynomial, \cite[pp.1041-42]{AnJi1999}.

	The following lemma plays a key role in the proof of our results and generalizes \cite[Lemma 4.7]{APZ2023}, as it does not require staying away from the walls.
	\begin{lemma}\label{S4 Lemma ratios difference}
		Let $t\mapsto r(t)$ be a positive function such that 
		$r(t)\rightarrow \infty$ and $\frac{r(t)}{t}\rightarrow 0$ 
		as $t\rightarrow\infty$. Then,  for bounded $y\in{G}$ and for all $g$ such that $|g|\leq r(t)$, the following asymptotic behavior holds:
		\begin{align*}
			\frac{h_{t}(g^{-1}y)}{h_{t}(g^{-1})}
			-\frac{\varphi_{0}(g^{-1}y)}{\varphi_{0}(g^{-1})}\,
			=\,\mathrm{O}\Big(\frac{r(t)}{t}\Big)
			\qquad\textnormal{as}\,\,\,t\longrightarrow\infty.
		\end{align*}
	\end{lemma}

	\begin{proof}
		Assume that $|y|\le\xi$ for some positive constant $\xi$. Then, for $t$ large enough, we have
		\begin{align*}
			|(g^{-1}y)^{+}|\,&=\,d(g^{-1}yK, eK)\,=\,d(yK, gK)\,\le \, d(gK, eK)+d(yK, eK) \\
			&\le\,|g|+\xi\,
			<\,r(t)+r(t)\,	<\,2r(t).
		\end{align*}
		
		The asymptotics \eqref{S2 heat kernel critical region}  yield
		\begin{align}\label{difference}
			\frac{h_{t}(g^{-1}y)}{h_{t}(g^{-1})}
			-\frac{\varphi_{0}(g^{-1}y)}{\varphi_{0}(g^{-1})}\,
			\sim \, \frac{\varphi_{0}(g^{-1}y)}{\varphi_{0}(g^{-1})}
			\Big\lbrace{
				\frac{\mathbf{b}\left(-i\frac{(g^{-1}y)^{+}}{2t}\right)^{-1}} {\mathbf{b}\left(-i\frac{(g^{-1})^{+}}{2t}\right)^{-1}}e^{-\tfrac{|(g^{-1}y)^{+}|^{2}}{4t}+\tfrac{|(g^{-1})^{+}|^{2}}{4t}}-1
			}\Big\rbrace
		\end{align}
		On the one hand, the quotient $\frac{\varphi_{0}(g^{-1}y)}{\varphi_{0}(g^{-1})}$ is uniformly bounded  when $|g^{+}|\leq r(t)$ and $|y|\leq \xi$ due to \eqref{phi0bdd}. On the other hand, by the triangle inequality we get
		\begin{align}\label{Gaussian}
			e^{-\tfrac{|(g^{-1}y)^{+}|^{2}}{4t}+\tfrac{|(g^{-1})^{+}|^{2}}{4t}}\,
			=\,e^{\left(|(g^{-1})^{+}|-|(g^{-1}y)^{+}|\right)
				\tfrac{|(g^{-1})^{+}|+|(g^{-1}y)^{+}|}{4t}}\,
			=\,e^{\mathrm{O}\big(\tfrac{r(t)}{t}\big)}\,
			=\,1+\mathrm{O}\big(\tfrac{r(t)}{t}\big).
		\end{align}
		Finally, observe that for all $H\in \overline{\mathfrak{a}^{+}}$ such that $|H|=o(t)$, we have from \eqref{bfunction} and \eqref{bfunction derivative} that $$\mathbf{b}\big(-i\tfrac{H}{2t}\big)^{-1}\asymp 1, \quad 
		\left| \nabla_H
		\mathbf{b}\big(-i\tfrac{H}{2t}\big)^{-1}\right|\,
		=\,\mathrm{O}\left(\frac{1}{t}\right).$$
		Therefore, by the mean value theorem we get $$\left|\mathbf{b}\left(-i\frac{(g^{-1}y)^{+}}{2t}\right)^{-1} - \mathbf{b}\left(-i\frac{(g^{-1})^{+}}{2t}\right)^{-1}\right| \lesssim \frac{|y|}{t}= \mathrm{O}\left(\frac{1}{t}\right) $$
		where we used the fact that $d(xK,  yK) \geq |x^{+}-y^{+}|$ for all $x, y\in G$, \cite[Lemma 2.1.2]{AnJi1999}. In other words, 
		\begin{equation}\label{bquot}
			\frac{\mathbf{b}\left(-i\frac{(g^{-1}y)^{+}}{2t}\right)^{-1}} {\mathbf{b}\left(-i\frac{(g^{-1})^{+}}{2t}\right)^{-1}}=1+\mathrm{O}\left(\frac{1}{t}\right).
		\end{equation}
		In conclusion, \eqref{difference} yields by \eqref{Gaussian} and \eqref{bquot} that
		\begin{align*}
			\frac{h_{t}(g^{-1}y)}{h_{t}(g^{-1})}
			-\frac{\varphi_{0}(g^{-1}y)}{\varphi_{0}(g^{-1})}\,
			=\,\mathrm{O}\Big(\frac{r(t)}{t}\Big)
			\qquad\forall\,|g|\leq r(t),\,\,
			\forall\,|y|\leq \xi.
		\end{align*}
	\end{proof}

	\subsection{The Doob transform} Our main reference for this part is \cite[Chapter 9]{Gri2009}. Let $(\mathcal{M},\mu)$ be a Riemannian manifold. Any smooth positive function $h$ on $\mathcal{M}$ determines a
	new measure $\widetilde{\mu}$ on $\mathcal{M}$ by
	\begin{equation}\label{g9.20}
		\diff{\widetilde{\mu}}=h^2\diff{\mu}
	\end{equation}
	and, hence, a new weighted manifold $(\mathcal{M}, \widetilde{\mu}),$ whose heat kernel will be denoted by $\widetilde{h_t}$.   
	\begin{theorem} Let $h$ be a smooth positive function on $\mathcal{M}$ that satisfies
		the equation
		$\Delta_{\mu}h+ah=0$,
		where $a$ is a real constant. Then the following identities hold:
		\begin{equation}\label{g9.21}
			\Delta_{\widetilde{\mu}}=\frac{1}{h}\circ (\Delta_{\mu}+a\,\textnormal{Id}) \circ h,
		\end{equation}
		\begin{equation}\label{g9.22}
			e^{t\Delta_{\widetilde{\mu}}}=e^{at}\frac{1}{h}\circ e^{t\Delta}\circ h,\end{equation}
		\begin{equation}\label{g9.23}
			\widetilde{h_t}(x,y)=e^{at}\frac{h_t(x,y)}{h(x)h(y)}
		\end{equation}
		for all $t > 0$ and $x, y \in\mathcal{M} $.
	\end{theorem}
	In \eqref{g9.22} and \eqref{g9.23}, $h$ and $\frac{1}{h}$ are regarded as multiplication operators,
	the domain of the operators in \eqref{g9.21} is $C^{\infty}(\mathcal{M})$, and the domain of the
	operators in \eqref{g9.23} is $L^2 (\mathcal{M}, \widetilde{\mu})$.
	The change of measure \eqref{g9.20} satisfying \eqref{g9.21} and the associated opeartor change
	\eqref{g9.22} are referred to as Doob's $h$-transform, which is a useful tool to describe Markov processes conditioned upon certain events.

	\begin{remark}
		In \cite{APZ2023} the long-time asymptotics of the distinguished Laplacian were studied.  Recall that the solvable Lie group $S = (\exp \mathfrak{a})N = N(\exp \mathfrak{a})$  can be identified, as a manifold, with $\mathbb{X}= G/K$, owing to the Iwasawa decomposition $G = SK$. Consider orthonormal bases $\{H_j\}$ of $\mathfrak{a}$ and $\{X_k\}$ of $\mathfrak{n}$, respectively, and the left-invariant vector
		fields agreeing with the respective elements of the tangent space at the identity. Then the distinguished Laplacian is the operator
		$\sum\limits_{j} H_j^2 + 2\sum\limits_{k} X_k^2$.
		
		However, it can be also seen as the Doob transform induced by another eigenfuntion, the modular function of $S$,  defined by
		\begin{align*}
			\widetilde{\delta}(g)\,
			=\,\widetilde{\delta}(n(\exp{A}))\,
			=\,e^{-2\langle{\rho,A}\rangle}
			\qquad\forall\,g\in{S}.
		\end{align*}
		Here $n=n(g)$ and $A=A(g)$ denotes respectively the $N$-component and the
		$\mathfrak{a}$-component of $g$ in the Iwasawa decomposition.
	\end{remark}
	
	\subsection{The infinite Brownian loop}
	Anker, Bougerol and Jeulin \cite{ABJ02} introduced the notion of the infinite Brownian loop, as a means to overcome the fact that the asymptotic properties of the Brownian motion for manifolds with spectral gap often fail to give information for their geometry at infinity.  We next define this process on a general Riemannian manifold $\mathcal{M}$.
	
	Fix a point $x_0\in \mathcal{M}$. The Brownian bridge $X^{(L)}$
	around a of length $L > 0$ is the Brownian motion $\{X_t, \; 0 \leq t \leq  L\}$ on $\mathcal{M}$
	conditioned by $X_0 = X_L = x_0$.
	\begin{definition}
		The infinite Brownian loop $(X_t^0)$ around $x_0$ is, when it exists, the limit in distribution of the Brownian bridge $B^{(L)}$ as $L \rightarrow +\infty$. 
	\end{definition}
	In fact, in \cite{ABJ02}  it is proved that on symmetric spaces of dimension at infinity $\nu$,  this process not only exists and has no spectral gap, but the behavior of the radial part of this process
	at infinity is the same as the one of the Brownian motion in a $\nu$-dimensional
	Euclidean space. In addition, the infinite Brownian loop in this
	case is the relativized $\varphi_0$-process. Notice that since the action of
	$G$ is transitive on $\mathbb{X}=G/K$, there is no restriction to study the infinity Brownian loop only around the origin $x_0=\{K\}$.
	
	The processes relativized by a ground state were introduced on general Riemannian manifolds by Sullivan in \cite{Sul79} and \cite{Sul87}. In general there are many
	positive ground states. The interesting feature of the infinite Brownian loop
	is that it chooses in a canonical way one of them, which is arguably in a
	sense the most symmetric one. According to Davies \cite{Dav97}, the idea of studying the heat kernel by using this relativized process
	goes back to Nelson \cite{Nel66} and Gross \cite{G75}, at
	least when there is a unique ground state. It plays a major role in Davies and
	Simon \cite{DS84} for instance.
	
	\section{The relativized Laplacian}\label{section3}
	By the discussion in Subsections 2.3 and 2.4, let us take in $\mathbb{X}$ as $h$ the ground spherical function $\varphi_0$ with $a=|\rho|^2$. 
	Notice that $\Delta_{\widetilde{\mu}}$
	is self-adjoint with respect to the measure $\varphi_0(x)^2\diff(xK)$ which has polynomial growth, due to \eqref{S2 estimate of delta} and \eqref{S2 global estimate phi0}:
	\begin{align*}\int_{G/K}\diff(xK)\,\varphi_0(x)^2 \,f(xK)&\asymp \int_{\mathfrak{a}^{+}}\diff{x^{+}}\,\Big\lbrace\prod_{\alpha\in\Sigma^{+}}
		\Big( 
		\frac{\langle\alpha,x^{+}\rangle}
		{1+\langle\alpha,x^{+}\rangle}
		\Big)^{m_{\alpha}}\Big\rbrace\,
		\Big\lbrace \prod_{\alpha\in\Sigma_{r}^{+}} 
		\left(1+\langle\alpha,x^{+}\rangle\right)\Big\rbrace^2 \\
		&\times \int_{K}\diff{k}\,f(k(\exp x^{+})K).
	\end{align*}
	Therefore, the corresponding heat kernel, 
	in the sense that 
	$$(e^{t\Delta_{\widetilde{\mu}}}f)(xK)=\int_{G/K}\diff{\widetilde{\mu}}({yK})\,\widetilde{{h_t}}(xK,yK)f(yK),$$
	is given by
	$$	\widetilde{{h_t}}(xK,yK)=e^{|\rho|^2t}\frac{h_t(xK,yK)}{\varphi_0(x)\varphi_0(y)}=e^{|\rho|^2t}\frac{h_t(y^{-1}x)}{\varphi_0(x)\varphi_0(y)},$$
	and it is bi-$K$-invariant in the sense that $\widetilde{{h_t}}(kxK,kyK)=\widetilde{{h_t}}(xK,yK)$, for all $x, y\in G$, $k\in K$.
	
	\begin{remark}
		Notice that $\widetilde{{h_t}}\diff{\widetilde{\mu}}$ is a probability measure
		on $\mathbb{X}=G/K$. Indeed, we have
		\begin{align*}
			\int_{G/K}\diff{\widetilde{\mu}}{(yK)}\,\widetilde{{h_t}}(xK, yK)&=e^{|\rho|^2t}\frac{1}{\varphi_{0}(x)}\int_{G}\diff{y}\,\varphi_0(y){h}_{t}(y^{-1}x)\\
			&=e^{|\rho|^2t}\frac{1}{\varphi_{0}(x)}\int_{G}\diff{z}\,\varphi_0(xz^{-1}){h}_{t}(z)\\
			&=e^{|\rho|^2t}\frac{1}{\varphi_{0}(x)}\int_{G}\diff{z}\,{h}_{t}(z)\int_{K}\diff{k}\,\varphi_0(xkz^{-1})\\
			&=e^{|\rho|^2t}\int_{G}\diff{z}\,{h}_{t}(z)\varphi_0(z), 
		\end{align*}
		by  \eqref{phi0split} and the fact that $\varphi_{0}(z^{-1})=\varphi_{0}(z)$. We conclude by observing that by \eqref{S2 heat kernel inv}, the last integral equals $(\mathcal{H}h_t)(0)=e^{-|\rho|^2t}$. 
	\end{remark}

	The first subsection is devoted to determine the critical region 
	where the heat kernel $\widetilde{h}_{t}$ concentrates. 
	In the next two subsections, we study respectively the $L^{1}$ 
	and the $L^{\infty}$ asymptotic convergence of solutions to \eqref{S1 HE S}
	with compactly supported initial data.
	We discuss the same questions for other initial data in the last subsection.
	
	%%%%%%%%%%%%%%%%%%%%%%%%%%%%%%%%%%%%%%%%%%%%%%%%%%%%%%%%%%%%%%%%%%%%%%%%%%%%%%%%
	\subsection{Asymptotic concentration of the relativized heat kernel}
	The following proposition determines where the heat kernel $\widetilde{{h_t}}$ 
	concentrates, in the sense that the integral of the heat kernel  outside this region tends to $0$, as time grows to infinity (see Figure 1).

	Define the function
	$\mu:\overline{\mathfrak{a}^{+}}\rightarrow\mathbb{R}^{+}$ by $\mu(H)=\min_{\alpha\in\Sigma^{+}}\langle{\alpha,H}\rangle$.
	
	\begin{proposition}
		Let $t\mapsto\varepsilon(t)$ be a positive function such that 
		$\varepsilon(t)\searrow0$ and $\varepsilon(t)\sqrt{t}\rightarrow\infty$ 
		as $t\rightarrow\infty$. 
		Then the heat kernel associated with the relativized Laplacian on $(\mathbb{X}, \widetilde{\mu})$
		concentrates asymptotically in $K(\exp\Omega_{t})$, where
		\begin{align*}
			\Omega_{t}\,
			=\,\big\lbrace{
				H\in\overline{\mathfrak{a}^{+}}\,|\,
				\varepsilon(t)\sqrt{t}\le|H|\le\tfrac{\sqrt{t}}{\varepsilon(t)}
				\,\,\,\textnormal{and}\,\,\,
				\mu(H)\ge\varepsilon(t)\sqrt{t}
			}\big\rbrace.
		\end{align*}
		In other words,
		\begin{align*}
			\lim_{t\rightarrow\infty}
			\int_{gK\,\textrm{s.t.}\,
				g^{+}\in\overline{\mathfrak{a}^{+}}
				\smallsetminus\Omega_{t}}
			d\widetilde{\mu}(gK)\,\widetilde{h_t}(gK)\,
			=\,0
		\end{align*}
		where $g^{+}$ denotes the middle component of $g$ in the Cartan decomposition.
	\end{proposition}
	
	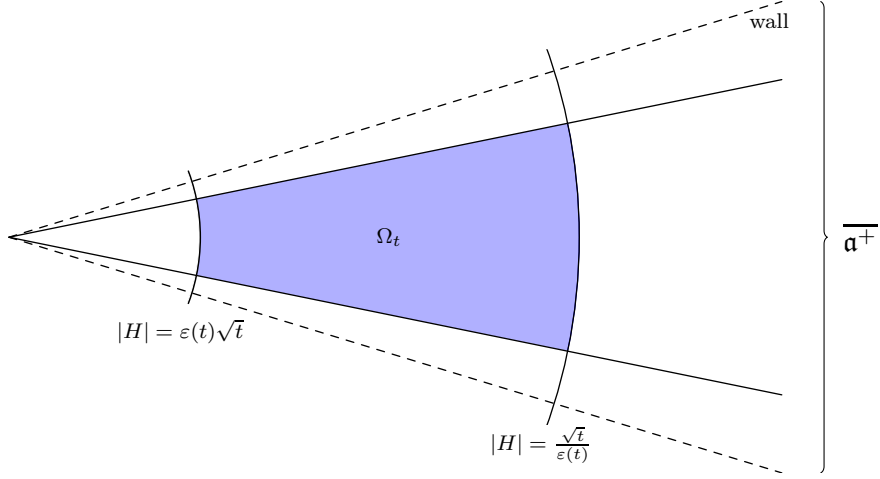
\begin{figure}
		\centering
		\begin{tikzpicture}[line cap=round,line join=round,>=triangle 45,x=1cm,y=1cm,scale=2.5]
\clip(-0.5,-1.25) rectangle (5,1.25);
\draw [line width=0.0pt,color=blue,fill=blue,fill opacity=0.31] (0,0) circle (3cm);
\fill[line width=0pt,color=white,fill=white,fill opacity=1] (0,0) -- (4.0654939275664805,0.8360139301518585) -- (0,4) -- (-4,1) -- (-4,-1) -- (0,-4) -- (4.183307817025808,-0.8586935566861592) -- cycle;
\draw [line width=1.5pt] (0,0) circle (1cm);
\draw [line width=0.5pt] (0,0) circle (3cm);   
\fill[line width=0pt,color=white,fill=white,fill opacity=1] (0,0) -- (4,1.5) -- (0,4) -- (-4,2) -- (-4,-2) -- (0,-4) -- (4,-1.5) -- cycle;
\draw [line width=0pt,color=white,fill=white,fill opacity=1] (0,0) circle (1cm);

\draw [line width=0.5pt] (0,0)-- (4.0654939275664805,0.8360139301518585);
\draw [line width=0.5pt] (0,0)-- (4.0654939275664805,-0.8360139301518585);
%\draw [line width=0.5pt] (0,0)-- (4.5,0);
\draw [color=white,fill=white,fill opacity=1] (1.5,-1) rectangle (2,-1);
\draw [color=white,fill=white,fill opacity=1] (1.5,1) rectangle (2,1);
\draw [color=white,fill=white,fill opacity=1] (2.5,-1) rectangle (3,-1.5);
\draw [color=white,fill=white,fill opacity=1] (2.5,1) rectangle (3,1.5);
\draw [dashed, line width=0.5pt] (0,0)-- (4.0654939275664805,-1.25);
\draw [dashed, line width=0.5pt] (0,0)-- (4.0654939275664805,1.25);
\draw[decoration={brace,mirror,raise=5pt},decorate]
  (4.2,-1.25) -- node[right=6pt] {$\,\,\overline{\mathfrak{a}^{+}}$} (4.2,1.25);

\begin{scriptsize}
\draw (0.9,-0.5) node{$|H|=\varepsilon(t)\sqrt{t}$};
\draw (2.8,-1.1) node{$|H|=\tfrac{\sqrt{t}}{\varepsilon(t)}$};
\draw (2,0) node{$\Omega_{t}$};
%\draw (4.4,-0.1) node{$\rho$};
\draw (4,1.15) node{wall};
\end{scriptsize}
\end{tikzpicture}
		\caption{$L^1$ critical region $\Omega_{t}$ 
			in the positive Weyl chamber.}
		\label{fig concentration2}
	\end{figure}
	
	\begin{proof}
		Observe that
		\begin{align*}
			I(t)\,
			=\,\int_{K(\exp\Omega_{t})} d\widetilde{\mu}(gK)\,\widetilde{{h}_{t}}(gK)\,
			=\,e^{|\rho|^{2}t}\int_{K(\exp\Omega_{t})K}\diff{g}\,\varphi_{0}(g)\,h_{t}(g).
		\end{align*}
		Then the  claim follows by resuming the computations in \cite[Proposition 4.2]{APZ2023}.
	\end{proof}

	\subsection{Heat asymptotics in $L^1$ for compactly supported initial data}
	In this subsection, we investigate the long-time asymptotic convergence in 
	$L^1(\widetilde{\mu})$ of solutions to the Cauchy problem \eqref{S1 HE S}
	where the initial data $f$ is assumed continuous and compactly supported, say $|y|\leq \xi$.  
	
	The mass function is defined by
	\begin{align}\label{massNEW}
		\widetilde{M}(g)\,
		=\,\frac{(\frac{f}{\varphi_0}*\varphi_{0})
		}{\varphi_{0}}(g)=\frac{1}{\varphi_{0}(g)}\int_{G}\diff{y}\,f(y)\varphi_{0}(y)\varphi_0(y^{-1}g)
		\qquad\forall\,g\in{G}.
	\end{align}

	\begin{remark}
		If $f\in\mathcal{C}_{c}(\mathbb{X})$,
		then the mass function $\widetilde{M}$ is bounded.
		This follows indeed from the local Harnack inequality in \eqref{phi0bdd}.
	\end{remark}
	
	\begin{remark}\label{rmk mass}
		If $f$ is bi-$K$-invariant,
		we notice  that
		\begin{align}\label{S4 radial mass}
			\frac{( \frac{f}{\varphi_0} *\varphi_{0})(g)}{\varphi_{0}(g)}\,
			&=\,\tfrac{1}{\varphi_{0}(g)}\,
			\int_{G}\diff{y}\,f(y)\varphi_{0}(y)\varphi_{0}(y^{-1}g)\notag\\[5pt]
			&=\,\tfrac{1}{\varphi_{0}(g)}\,
			\int_{G}\diff{y}\,f(y)\varphi_{0}(y)\,
			\underbrace{\vphantom{\Big|}
				\int_{K}\diff{k}\,\varphi_{0}(y^{-1}kg)
			}_{\varphi_{0}(g)\,\varphi_{0}(y)} \notag\\
			&=\,\int_{G}\diff{y}\,f(y)\,\varphi_{0}(y)^2=\int_{\mathbb{X}}\diff{\widetilde{\mu}}(yK)\,f(yK).
		\end{align}
		Hence the mass function $\widetilde{M}$ is a constant 
		if $f$ is bi-$K$-invariant and $f$ belongs to $L^{1}(\widetilde{\mu})$, thus generalizing the Euclidean mass.
	\end{remark}

	\vspace{5pt}
	Now, let us prove the first part of \cref{S1 Main thm 2}.
	\begin{proof}[Proof of \eqref{S1 L1 disting} in \cref{S1 Main thm 2}]
		First of all, the solution $u$ to \eqref{S1 HE S} writes as 
		\begin{align}
			u(t,gK)\,&=(e^{t\Delta_{\widetilde{\mu}}}f)(gK)=\int_{G/K}\diff{\mu}({yK})\,\varphi_0(y)^2\widetilde{{h_t}}(gK,yK)f(yK) \notag\\
			&=e^{|\rho|^2t}\frac{1}{\varphi_{0}(g)}\int_{G}\diff{y}\,\varphi_0(y){h_t}(y^{-1}g)f(yK)\notag\\
			&=\widetilde{h_t}(g)\int_{G}\diff{y}\,\varphi_0(y)\frac{h_t(y^{-1}g)}{h_t(g)}f(yK). \label{solNEW}
		\end{align}
		Then, by \eqref{massNEW} and \eqref{solNEW}, we have
		\begin{align}\label{S4 difference}
			u(t,g)-\widetilde{M}(g)\,\widetilde{h}_{t}(g)\,
			&=\,\widetilde{h}_{t}(g)\,
			\int_{G}\diff{y}\,\varphi_0(y)f(yK)\,
			\Big\lbrace{
				\frac{h_{t}(y^{-1}g)}{h_{t}(g)}
				-\frac{\varphi_{0}(y^{-1}g)}{\varphi_{0}(g)}
			}\Big\rbrace\notag\\[5pt]
			&=\,\widetilde{h}_{t}(g)\,
			\int_{G}\diff{y}\,\varphi_0(y)f(yK)\,
			\Big\lbrace{
				\frac{h_{t}(g^{-1}y)}{h_{t}(g^{-1})}
				-\frac{\varphi_{0}(g^{-1}y)}{\varphi_{0}(g^{-1})}
			}\Big\rbrace
		\end{align}
		where the last expression is derived from the symmetries 
		$h_{t}(x^{-1})=h_{t}(x)$ and $\varphi_{0}(x^{-1})=\varphi_{0}(x)$.
		Applying \cref{S4 Lemma ratios difference} for $r(t)=\frac{\sqrt{t}}{\varepsilon(t)}$, we have
		\begin{align*}
			\frac{h_{t}(g^{-1}y)}{h_{t}(g^{-1})}
			-\frac{\varphi_{0}(g^{-1}y)}{\varphi_{0}(g^{-1})}\,
			=\,\mathrm{O}\Big(\frac{1}{\varepsilon(t)\sqrt{t}}\Big)
			\qquad\forall\,g\in{K(\exp\Omega_{t})K},\,\,
			\forall\,y\in\supp{f,}
		\end{align*}
		and therefore the integral of
		$u(t,\cdot)-\widetilde{M}\,\widetilde{h}_{t}$ 
		over the critical region
		\begin{align*}
			\int_{\mathbb{X}\cap{K(\exp\Omega_{t})}} \hspace*{-.2cm}
			\diff{\widetilde{\mu}(gK)}\,    
			|u(t,gK)-\widetilde{M}(g)\widetilde{h}_{t}(gK)|\,
			\lesssim\,
			\frac{1}{\varepsilon(t)\sqrt{t}}\,
			\underbrace{\vphantom{\Big|}
				\int_{\mathbb{X}}\diff{\widetilde{\mu}(gK)}\,\widetilde{h}_{t}(gK)
			}_{1}\,
			\underbrace{\vphantom{\Big|}
				\int_{\supp f}\diff{\widetilde{\mu}(yK)}\,\left|\frac{f(yK)}{\varphi_0(y)}\right|
			}_{\const}
		\end{align*}
		tends asymptotically to $0$. It remains to check that outside the critical region, the integral
		\begin{align*}
			\int_{\mathbb{X}\smallsetminus{K(\exp\Omega_{t})}} \hspace*{-.2cm}
			\diff{\widetilde{\mu}(gK)}\,    
			|u(t,gK)-\widetilde{M}(g)\widetilde{h}_{t}(gK)|\,
			&\le\,
			\int_{\mathbb{X}\smallsetminus{K(\exp\Omega_{t})}} \hspace*{-.2cm}
			\diff{\widetilde{\mu}(gK)}\,
			|u(t,gK)|\\[5pt]
			&+\,
			\int_{\mathbb{X}\smallsetminus{K(\exp\Omega_{t})}} \hspace*{-.2cm}
			\diff{\widetilde{\mu}(gK)}\,    
			|\widetilde{M}(g)|\widetilde{h}_{t}(gK)
		\end{align*}
		tends also to $0$. On the one hand, we know that $\widetilde{M}$ is bounded
		and that the heat kernel $\widetilde{h}_{t}$ asymptotically concentrates in 
		$K(\exp\Omega_{t})K$, hence
		\begin{align*}
			\int_{\mathbb{X}\smallsetminus{K(\exp\Omega_{t})}} \hspace*{-.2cm}
			\diff{\widetilde{\mu}(gK)}\,
			|\widetilde{M}(g)|\widetilde{h}_{t}(gK)\,
			\longrightarrow\,0
		\end{align*}
		as $t\rightarrow\infty$. On the other hand, notice that for all
		$y\in\supp{f}$ and for all $g\in{G}$ such that
		$g^{+}\notin\Omega_{t}$, we have 
		\begin{align}\label{S4 Omega''}
			(y^{-1}g)^{+}\,\notin\,\Omega_{t}''
			=\,
			\big\lbrace{
				H\in\overline{\mathfrak{a}^{+}}\,|\,
				\varepsilon''(t)\sqrt{t}\le|H|\le\tfrac{\sqrt{t}}{\varepsilon''(t)}
				\,\,\,\textnormal{and}\,\,\,
				\mu(H)\ge\varepsilon''(t)\sqrt{t}
			}\big\rbrace
		\end{align}
		where $\varepsilon''(t)=2\varepsilon(t)$ (see for instance the proof of \cite[Lemma 4.7]{APZ2023}).
		Hence
		\begin{align*}
			\int_{\mathbb{X}\smallsetminus{K(\exp\Omega_{t})}} \hspace*{-.2cm}
			\diff{\widetilde{\mu}(gK)}\,
			|u(t,gK)|\,
			&\le\,\int_{\mathbb{X}}\diff{\widetilde{\mu}(yK)}\,|f(yK)|\,
			\int_{\mathbb{X}\smallsetminus{K(\exp\Omega_{t})}} \diff{\widetilde{\mu}(gK)}\, \widetilde{{h_t}}(gK, yK) 
			\\[5pt]
			&\lesssim\,
			\underbrace{\vphantom{
					\int_{\mathbb{X}\smallsetminus{K(\exp\Omega_{t}'')}}}
				\int_{\mathbb{X}}\diff{\widetilde{\mu}(yK)}\,|f(yK)|
			}_{\|f\|_{L^{1}(\widetilde{\mu})}}\,
			\underbrace{
				\int_{\mathbb{X}\smallsetminus{K(\exp\Omega_{t}'')}}
				\diff{\widetilde{\mu}(gK)}\,\widetilde{h}_{t}(gK)
			}_{\longrightarrow\,0}.\
		\end{align*}
		This concludes the proof of the heat asymptotics in $L^{1}(\widetilde{\mu})$ for the 
		relativized Laplacian $\Delta_{\widetilde{\mu}}$ on $\mathbb{X}$ and for initial data
		$f\in\mathcal{C}_{c}(\mathbb{X})$.
	\end{proof}

	%%%%%%%%%%%%%%%%%%%%%%%%%%%%%%%%%%%%%%%%%%%%%%%%%%%%%%%%%%%%%%%%%%%%%%%%%%%%%%%%
	\subsection{Heat asymptotics in $L^{\infty}$
		for compactly supported initial data}

	In the following two propositions we collect some elementary properties of the 
	relativized heat kernel. The first one clarifies the lower
	and the upper bounds of $\widetilde{h}_{t}$, while the second one describes
	its critical region for the $L^{\infty}$ norm (see Figure 2).

	\begin{proposition}\label{S4 htilde estimate proposition}
		The heat kernel $\widetilde{h}_{t}$ associated with the relativized Laplacian
		satisfies
		\begin{align}\label{S4 htilde estimate}
			\|\widetilde{h}_{t}\|_{L^{\infty}(\widetilde{\mu})}\,
			\asymp\,t^{-\frac{\nu+n}{4}}
		\end{align}
		for $t$ large enough.
	\end{proposition}
	
	\begin{proof}
		Using the global estimate
		\eqref{S2 heat kernel}, we have for $t$ large
		\begin{align}
			\widetilde{h}_{t}(g)\,
			&\asymp\,
			t^{-\frac{n}{2}}\,
			\Big\lbrace{
				\prod\nolimits_{\alpha\in\Sigma_{r}^{+}}
				\big(1+t+\langle{\alpha,g^{+}}\rangle\big)^{
					\tfrac{m_{\alpha}+m_{2\alpha}}{2}-1}
			}\Big\rbrace\,
			e^{-\frac{|g^{+}|^2}{4t}} \notag \\
			&\asymp\,
			t^{-\frac{\nu+n}{4}}\,
			\Big\lbrace{  	\prod\nolimits_{\alpha\in\Sigma_{r}^{+}}
				\big(1+\frac{\langle{\alpha,g^{+}}\rangle}{\sqrt{t}}\big)^{
					\tfrac{m_{\alpha}+m_{2\alpha}}{2}-1}
				\Big\rbrace}\,
			e^{-\frac{|g^{+}|^2}{4t}} \label{S4 htilde1}
		\end{align}
		The lower bound in \eqref{S4 htilde estimate}
		is taken 	by evaluating the right hand side of \eqref{S4 htilde1} at
		$g_{0}=\exp(\sqrt{t}\rho)$.
		For the upper bound, 
		we get from \eqref{S4 htilde1} that
		\begin{align}\label{S4 htilde1''}
			\widetilde{h}_{t}(g)\,
			\lesssim\,t^{-\frac{\nu+n}{4}}\,
			\underbrace{
				e^{-\frac{1}{4}\big|\tfrac{g^{+}}{\sqrt{t}}\big|^{2}}
				\prod\nolimits_{\alpha\in\Sigma_{r}^{+}}
				\big(1+\langle{\alpha,\tfrac{g^{+}}{\sqrt{t}}}\rangle\big)^{
					\tfrac{m_{\alpha}+m_{2\alpha}}{2}-1}
			}_{=\,\textrm{O}(1)}\,
			\lesssim\,t^{-\frac{\nu+n}{4}}.
		\end{align}
	\end{proof}
	
	\begin{figure}
		\centering
		\begin{tikzpicture}[line cap=round,line join=round,>=triangle 45,x=1cm,y=1cm,scale=2.5]
[line cap=round,line join=round,>=triangle 45,x=1cm,y=1cm,scale=2.5]
\clip(-0.5,-1.25) rectangle (5,1.25);
\draw [line width=0.0pt,color=blue,fill=blue,fill opacity=0.31] (0,0) circle (3cm);

\draw [line width=0.5pt] (0,0) circle (3cm);   
\fill[line width=0pt,color=white,fill=white,fill opacity=1] (0,0) -- (4.9,1.5) -- (0,4.9) -- (-4.9,2) -- (-4.9,-2) -- (0,-4.9) -- (4.9,-1.5) -- cycle;

	\draw [dashed, line width=0.5pt] (0,0)-- (4.0654939275664805,-1.25);
	\draw [dashed, line width=0.5pt] (0,0)-- (4.0654939275664805,1.25);
	\draw[decoration={brace,mirror,raise=5pt},decorate]
	(4.2,-1.25) -- node[right=6pt] {$\,\,\overline{\mathfrak{a}^{+}}$} (4.2,1.25);

	\begin{scriptsize}
	
		\draw (2.8,-1.1) node{$|H|=\tfrac{\sqrt{t}}{\varepsilon(t)}$};
		\draw (2,0) node{$R_{t}$};
		%\draw (4.4,-0.1) node{$\rho$};
		\draw (4,1.15) node{wall};
	\end{scriptsize}
\end{tikzpicture}
		\caption{$L^{\infty}$ critical region $R_{t}$ in the positive Weyl chamber.}
		\label{fig concentration3}
	\end{figure}
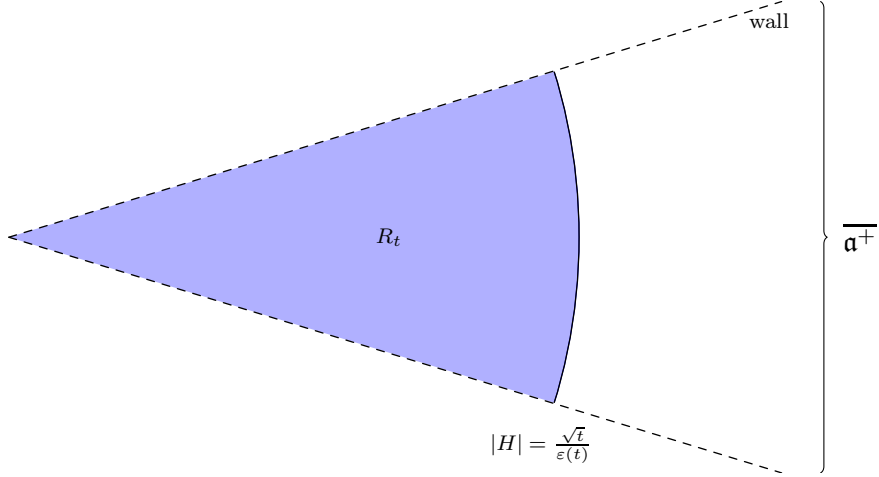
	
	Unlike the heat kernel of the distinguished Laplacian, the $L^1$ and $L^{\infty}$ critical regions do not coincide.

	\begin{proposition} \label{Linftyconc}
		The heat kernel $\widetilde{h}_{t}$ concentrates asymptotically
		in the critical region 
		$$K\exp(R_t), \quad \textnormal{where } R_t=\left\{H\in \overline{\mathfrak{a}^{+}}: \; |H|\leq \frac{\sqrt{t}}{\varepsilon(t)} \right\}.$$
		for the $L^{\infty}$ norm.
		In other words,
		\begin{align*}
			t^{\frac{\nu+n}{4}}\,
			\|\widetilde{h}_{t}\|_{L^{\infty}
				(\mathbb{X}\smallsetminus{K(\exp R_{t})})}\,
			\longrightarrow\,0
			\qquad\textnormal{as}\,\,\,t\rightarrow\infty.
		\end{align*}
	\end{proposition}
	
	\begin{proof}
		Let us study the sup norm of $\widetilde{h}_{t}$ outside the critical region. 
		
		If $|g^{+}|>\tfrac{\sqrt{t}}{\varepsilon(t)}$, then we deduce from \eqref{S4 htilde1} that 
		for any $N>0$,
		\begin{align*}
			t^{\frac{\nu+n}{4}}\,
			\widetilde{h}_{t}(g)\,
			\lesssim\,
			e^{-\frac{1}{4}\big|\tfrac{g^{+}}{\sqrt{t}}\big|^{2}}
			\prod\nolimits_{\alpha\in\Sigma_{r}^{+}}
			\big(1+\langle{\alpha,\tfrac{g^{+}}{\sqrt{t}}}\rangle\big)^{
				\tfrac{m_{\alpha}+m_{2\alpha}}{2}-1}\,
			\lesssim\,\varepsilon(t)^{N},
		\end{align*}
		which tends to $0$.
		Last, observe that for all $g$ such that $|g^{+}|=O(\sqrt{t})$, we have 
		$t^{\frac{\nu+n}{4}}\widetilde{h_t}(g)\gtrsim 1.$
	\end{proof}
	
	\vspace{5pt}
	Finally, let us prove the remaining part of \cref{S1 Main thm 2}.
	\begin{proof}[Proof of \eqref{S1 Linf disting} in \cref{S1 Main thm 2}]
		As before, let us write 
		\begin{align}
			|u(t,g)-\widetilde{M}(g)\widetilde{h}_{t}(g)|\,\leq \widetilde{h}_{t}(g)\,
			\int_{G}\diff{y}\,\varphi_0(y)|f(yK)|\,
			\left|{
				\frac{h_{t}(g^{-1}y)}{h_{t}(g^{-1})}
				-\frac{\varphi_{0}(g^{-1}y)}{\varphi_{0}(g^{-1})}
			}\right|
		\end{align}
		In the critical region $K(\exp R_{t})$, we have, for $\supp f  \subseteq B(eK, \xi)$,
		\begin{align*}
			|u(t,g)-\widetilde{M}(g)\widetilde{h}_{t}(g)|\,
			\le\,\widetilde{h}_{t}(g)\,
			\int_{|y|\leq \, \xi}\diff{y}\,\,\varphi_0(y)\,|f(yK)|\,
			\Big|{
				\frac{h_{t}(g^{-1}y)}{h_{t}(g^{-1})}
				-\frac{\varphi_{0}(g^{-1}y)}{\varphi_{0}(g^{-1})}
			}\Big|
		\end{align*}
		with
		\begin{align*}
			\Big|{
				\frac{h_{t}(g^{-1}y)}{h_{t}(g^{-1})}
				-\frac{\varphi_{0}(g^{-1}y)}{\varphi_{0}(g^{-1})}
			}\Big|\,
			\lesssim\,\frac{1}{\varepsilon(t)\sqrt{t}}
		\end{align*}
		according to \eqref{S4 difference} 
		and to \cref{S4 Lemma ratios difference} for $r(t)=\frac{\sqrt{t}}{\varepsilon(t)}$.
		Then we deduce from \eqref{S4 htilde estimate} that
		\begin{align*}
			t^{\frac{\nu+n}{4}}\,
			|u(t,g)-\widetilde{M}(g)\widetilde{h}_{t}(g)|\,
			\lesssim\,\tfrac{1}{\varepsilon(t)\sqrt{t}}
			\qquad\forall\,g\in K(\exp R_{t})K
		\end{align*}
		where the right-hand side tends to $0$ as $t\rightarrow\infty$.
		Outside the critical region, we estimate separately $u(t,g)$
		and $\widetilde{M}(g)\widetilde{h}_{t}(g)$. On the one hand, we know that
		$\widetilde{M}(g)$ is a bounded function and that 
		$\widetilde{h}_{t}(g)=\mathrm{o}(t^{-({\nu+n})/4})$. Then 
		$t^{(\nu+n)/4}\widetilde{M}(g)\widetilde{h}_{t}(g)$ 
		tends to $0$ as $t\rightarrow\infty$.
		On the other hand, since $|g^{+}|>\frac{\sqrt{t}}{\varepsilon(t)}$ and 
		$|y|\leq\xi$ imply that $|(g^{-1}y)^{+}|>\frac{\sqrt{t}}{\varepsilon(t)}-\xi>\frac{1}{2}\frac{\sqrt{t}}{\varepsilon(t)}$, we obtain
		\begin{align*}
			|u(t,g)|\,
			\lesssim\,\int_{G}\diff{\widetilde{\mu}}(yK)\,
			|f(yK)|\,|\widetilde{h}_{t}(gK, yK)|\,
		\end{align*}
		which is $\textrm{o}(t^{-(\nu+n)/4})$ (see Case 1 in \cref{Linftyconc}).
		In conclusion,
		\begin{align*}
			t^{\frac{\nu+n}{4}}
			\|u(t,\,\cdot\,)-
			\widetilde{M}\,\widetilde{h}_{t}\|_{L^{\infty}(\widetilde{\mu})}\,
			\longrightarrow\,0
		\end{align*}
		as $t\rightarrow\infty$.
	\end{proof}

	By convexity we obtain easily the corresponding result for the $L^{p}$ norm.
	\begin{corollary}
		The solution $u$ to the Cauchy problem \eqref{S1 HE S} with initial 
		data $f\in\mathcal{C}_{c}(\mathbb{X})$ satisfies
		\begin{align}\label{S4 Lp disting}
			t^{\frac{u+n}{4p'}}
			\|u(t,\,\cdot\,)-
			\widetilde{M}\,\widetilde{h}_{t}\|_{L^{p}(\widetilde{\mu})}\,
			\longrightarrow\,0
			\qquad\textnormal{as}\quad\,t\rightarrow\infty,
		\end{align}
		for all $1<p<\infty$.
	\end{corollary}
	
	%%%%%%%%%%%%%%%%%%%%%%%%%%%%%%%%%%%%%%%%%%%%%%%%%%%%%%%%%%%%%%%%%%%%%%%%%%%%%%%%
	\subsection{Heat asymptotics for other initial data}\label{Subsect other data}
	We have obtained above the long-time asymptotic convergence in $L^{p}$
	($1\le{p}\le\infty$) for the relativized heat equation 
	with compactly supported initial data.
	It is natural and interesting to ask whether the expected convergence
	still holds when the initial data lie in larger functional spaces.
	The following corollaries give positive examples,
	but the optimal answer remains open.
	
	\begin{corollary}
		The asymptotic convergences \eqref{S1 L1 disting} and \eqref{S1 Linf disting}, 
		hence \eqref{S4 Lp disting}, still hold with initial data 
		$f\,\in{L}^{1}(\widetilde{\mu})$ bi-$K$-invariant.
	\end{corollary}
	
	\begin{proof}
		Notice from  \eqref{S4 radial mass} that 
		the mass function is a constant equal to $\widetilde{M}=\int_{\mathbb{X}}\diff{\widetilde{\mu}}\,f$ under the present assumption.
		Let us start with the $L^{1}$ convergence. 
		Given $\varepsilon>0$, let
		$f_0\,\in\mathcal{C}_{c}(\mathbb{X})$,
		with $f_{0}$ bi-$K$-invariant, be such that $\|f
		-f_{0}\|_{L^{1}(\widetilde{\mu})}<\tfrac{\varepsilon}{3}$. Then the corresponding mass function $\widetilde{M}_0=\int_{\mathbb{X}}\diff{\widetilde{\mu}}\,f_0$
		is also a constant.
		We observe firstly that the solution to the Cauchy problem
		\begin{align*}
			\partial_{t}u_0(t,gK)\,
			=\,\Delta_{\widetilde{\mu}}u_0(t,gK),
			\qquad
			u_0(0,gK)\,=\,f_0(gK)
		\end{align*}
		satisfies
		\begin{align}\label{S4 density1}
			\|u(t,\,\cdot\,)-u_0(t,\,\cdot\,)\|_{L^{1}(\widetilde{\mu})}\,
			=\,\left\|\int_{\mathbb{X}}\diff\widetilde{\mu}(yK)(f-f_0)(y)\widetilde{h}_{t}(\cdot, y)
			\right\|_{L^{1}(\widetilde{\mu})} 
			\le\,\|f-f_{0}\|_{L^{1}(\widetilde{\mu})}\,
			\|\widetilde{h}_{t}\|_{L^{1}(\widetilde{\mu})}\,
			<\,\tfrac{\varepsilon}{3},
		\end{align}
		since $\|\widetilde{h}_{t}\|_{L^{1}(\widetilde{\mu})}=1$. Secondly, there exists $T>0$ such that for all $t\ge{T}$,
		\begin{align}\label{S4 density2}
			\|u_0(t,\,\cdot\,)
			-\widetilde{M}_0\,\widetilde{h}_{t}\|_{L^{1}(\widetilde{\mu})}\,
			<\,\tfrac{\varepsilon}{3}
		\end{align}
		according to \cref{S1 Main thm 2}.
		Under the bi-$K$-invariance assumption, we have from \eqref{S4 radial mass} that
		\begin{align*}
			\widetilde{M}-\widetilde{M}_0\,
			=\,\int_{\mathbb{X}}\diff{\widetilde{\mu}(gK)}\,(f(gK)-{f}_{0}(gK)).
		\end{align*}
		Hence, we have thirdly
		\begin{align}\label{S4 density3}
			\|(\widetilde{M}-\widetilde{M}_0)\,\widetilde{h}_{t}
			\|_{L^{1}(\widetilde{\mu})}\,
			\le\,\|f-f_{0}\|_{L^{1}(\widetilde{\mu})}\,
			\|\widetilde{h}_{t}\|_{L^{1}(\widetilde{\mu})}\,
			<\,\tfrac{\varepsilon}{3}.
		\end{align}
		In conclusion, by putting \eqref{S4 density1}, \eqref{S4 density2} and
		\eqref{S4 density3} altogether, we obtain
		\begin{align*}
			\|u(t,\,\cdot\,)
			-\widetilde{M}\,\widetilde{h}_{t}\|_{L^{1}(\widetilde{\mu})}\,
			<\,\varepsilon
		\end{align*}
		for all $\varepsilon>0$.
		Let us turn to the $L^{\infty}$ convergence.
		According to \eqref{S4 htilde estimate} and \cref{S1 Main thm 2}, we have this time 
		\begin{align*}
			t^{\frac{\nu+n}{4}}\,
			\|u(t,\,\cdot\,)-u_0(t,\,\cdot\,)
			\|_{L^{\infty}(\widetilde{\mu})}\,
			\le\,\|f-f_{0}\|_{L^{1}(\widetilde{\mu})}\,
			t^{\frac{\nu+n}{4}}\,
			\|\widetilde{h}_{t}\|_{L^{\infty}(\widetilde{\mu})}\,  
			\lesssim\,\tfrac{\varepsilon}{3},
		\end{align*}
		\begin{align*}
			t^{\frac{\nu+n}{4}}\,
			\|u_0(t,\,\cdot\,)
			-\widetilde{M}_{0}\,\widetilde{h}_{t}\|_{L^{\infty}(\widetilde{\mu})}\,
			<\,\tfrac{\varepsilon}{3},
		\end{align*}
		and 
		\begin{align*}
			t^{\frac{\nu+n}{4}}\,
			\|(\widetilde{M}-\widetilde{M}_{0})\,\widetilde{h}_{t}
			\|_{L^{\infty}(\widetilde{\mu})}\,
			\le\,
			\|f-f_{0}\|_{L^{1}(\widetilde{\mu})}\,
			t^{\frac{\nu+n}{4}}\,
			\|\widetilde{h}_{t}\|_{L^{\infty}(\widetilde{\mu})}\,
			\lesssim\,\tfrac{\varepsilon}{3}.
		\end{align*}
		Altogether,
		\begin{align*}
			t^{\frac{\nu+n}{4}}\,
			\|\widetilde{v}(t,\,\cdot\,)
			-\widetilde{M}\,\widetilde{h}_{t}\|_{L^{\infty}(\widetilde{\mu})}\,
			\longrightarrow\,0
		\end{align*}
		as $t\rightarrow\infty$.
		The $L^{p}$ convergence follows from interpolation.
	\end{proof}
	
	\begin{corollary}
		The asymptotic convergences \eqref{S1 L1 disting} and \eqref{S1 Linf disting}, 
		hence \eqref{S4 Lp disting}, 
		still hold under the assumption $\frac{f(gK)\,e^{\langle{\rho,g^{+}}\rangle}}{\varphi_{0}(g)}\in L^1({\widetilde{\mu}})$, i.e., 
		\begin{align}\label{S4 other data assumption}
			\int_{G}\diff{g}\,|f(gK)|\,\varphi_0(g)\,e^{\langle{\rho,g^{+}}\rangle}\,
			<\,\infty.
		\end{align}
	\end{corollary}
	
	\begin{proof}
		We first recall that according to \cite[Lemma 4.8]{APZ2023},  we have that for all $g\in G$,
		\begin{align}\label{S4 ineq Iwasawa Cartan}
			\langle{\rho,A(g)}\rangle\,
			\le\,\langle{\rho,g^{+}}\rangle
		\end{align}
		where $A(g)$ denotes the $\mathfrak{a}$-component of $g$ in the Iwasawa 
		decomposition and $g^{+}$ denotes its $\overline{\mathfrak{a}^{+}}$-component
		in the Cartan decomposition.

		Then, we have
		\begin{align*}
			\int_{G}\diff{g}\,|f(gK)|\varphi_{0}(g)^2 = \int_{G}\diff{g}\,|f(gK)|\varphi_{0}(g)\int_{K} \diff{k}\;e^{\langle \rho, A(kg)\rangle}
			\leq \int_{G}\diff{g}\,|f(gK)|\,\varphi_{0}(g)\,e^{\langle{\rho,g^{+}}\rangle}.
		\end{align*}
		Here, we used the inequality \eqref{S4 ineq Iwasawa Cartan} and
		the fact that $(ky)^{+}=y^{+}$ for all $k\in{K}$.	Hence, the assumption \eqref{S4 other data assumption} is indeed stronger than
		$f\in{L^{1}(\widetilde{\mu})}$. Under this assumption, the mass function is
		bounded:
		\begin{align*}
			|\widetilde{M}(g)|\,
			&\le\,\tfrac{1}{\varphi_{0}(g)}\,
			\int_{G}\diff{y}\,|f(yK)|\,\varphi_0(y)\,\varphi_{0}(y^{-1}g)\\[5pt]
			&=\,\tfrac{1}{\varphi_{0}(g)}\,
			\int_{G}\diff{y}\,|f(yK)|\,\varphi_0(y)\,\,\int_{K}\diff{k}\,
			e^{\langle{\rho,A(kg)}\rangle}e^{\langle{\rho,A(ky)}\rangle}\\[5pt]
			&\le\,
			\underbrace{\vphantom{\Big|}
				\tfrac{1}{\varphi_{0}(g)}\,
				\int_{K}\diff{k}\,e^{\langle{\rho,A(kg)}\rangle}
			}_{=\,1}
			\underbrace{\vphantom{\Big|}
				\int_{G}\diff{y}\,|f(yK)|\,\varphi_0(y)\,e^{\langle{\rho,y^{+}}\rangle}
			}_{=\,C}\,
			=\,C.
		\end{align*}

		For proving the $L^{1}$ and the $L^{\infty}$ convergences, 
		we argue again by density.
		Since $\frac{f(gK)\,e^{\langle{\rho,g^{+}}\rangle}}{\varphi_{0}(g)}\in L^1({\widetilde{\mu}})$
		there exists a function $\frac{f_0(gK)\,e^{\langle{\rho,g^{+}}\rangle}}{\varphi_{0}(g)}$ in
		$\mathcal{C}_{c}(\mathbb{X})$ such that 
		
		\begin{align*}
			\int_{\mathbb{X}}\diff{\widetilde{\mu}}(gK)\,
			\frac{|f(gK)-f_{0}(gK)|\,e^{\langle{\rho,g^{+}}\rangle}}{\varphi_0(g)}\,=\,\int_{\mathbb{X}}\diff{\mu}(gK)\,{\varphi_0(g)}\,
			|f(gK)-f_{0}(gK)|\,e^{\langle{\rho,g^{+}}\rangle}
			<\,\tfrac{\varepsilon}{3}
		\end{align*}
		for every $\varepsilon>0$. 
		Let $u_0=e^{t\Delta_{\widetilde{\mu}}}f_0$ be the corresponding
		solution to the relativized heat equation, and denote by 
		$\widetilde{M}_{0}(g)=\tfrac{(\frac{f_0}{\varphi_0}*\varphi_{0})(gK)}{\varphi_{0}(g)}$ 
		the corresponding mass of $f_{0}$. On the one hand, there exists $T>0$
		such that for all $t\geq T$, we have
		\begin{align*}\label{S4 density2}
			\|u_0(t,\,\cdot\,)
			-\widetilde{M}_{0}\,\widetilde{h}_{t}\|_{L^{1}(\widetilde{\mu})}\,
			<\,\tfrac{\varepsilon}{3}
		\end{align*}
		since $f_{0}\in\mathcal{C}_{c}(\mathbb{X})$. 
		On the other hand, for every $g\in{G}$, we have
		\begin{align*}
			|\widetilde{M}(g)-\widetilde{M}_{0}(g)|\,
			&\le\,\tfrac{1}{\varphi_{0}(g)}\,
			\int_{G}\diff{y}\,|f(yK)-f_{0}(yK)|\,\varphi_0(y)\,\varphi_{0}(y^{-1}g)\\[5pt]
			&\le\,\int_{G}\diff{y}\,\varphi_{0}(y)\,|f(yK)-f_{0}(gK)|\,
			e^{\langle{\rho,y^{+}}\rangle}\,
			<\,\tfrac{\varepsilon}{3}.
		\end{align*}
		We conclude by resuming the proof of the previous corollary.
	\end{proof}

	%%%%%%%%%%%%%%%%%%%%%%%%%%%%%%%%%%%%%%%%%%%%%%%%%%%%%%%%%%%%
	%%%%%%%%%%                END PAGES               %%%%%%%%%%
	%%%%%%%%%%%%%%%%%%%%%%%%%%%%%%%%%%%%%%%%%%%%%%%%%%%%%%%%%%%%

	\vspace{10pt}\noindent\textbf{Acknowledgements.}
	The author is grateful to J.-Ph. Anker for suggesting the problem and for making valuable comments on a first version of the manuscript.  The author is supported by the Hellenic Foundation for Research and Innovation, Project HFRI-FM17-1733.
	%%%%%%%%%%% Bibliography
	\printbibliography
	%%%%%%%%%%% Bibliography
	%Grigor’yan A., Analytic and geometric background of recurrence and non-explosion of the Brownian
	%motion on Riemannian manifolds, Bull. Amer. Math. Soc., 36 (1999) 135-249
	
	%%%%%%%%%%% Affiliation

	\vspace{10pt}
	\address{
		\noindent\textsc{Effie Papageorgiou:}
		\href{mailto:papageoeffie@gmail.com}
		{papageoeffie@gmail.com}\\
		Department of Mathematics and Applied Mathematics,
		University of Crete,
		Heraklion, Greece}
	
	%%%%%%%%%%%%%%%%%%%%%%%%%%%%%%%%%%%%%%%%%%%%%%%%%%%%%%%%%%%%
	%%%%%%%%%%%%%%%%%%%%%%%%%%%%%%%%%%%%%%%%%%%%%%%%%%%%%%%%%%%%
	%%%%%%%%%%%%%%%%%%%%%%%%%%%%%%%%%%%%%%%%%%%%%%%%%%%%%%%%%%%%

\end{document}